\documentclass[reqno]{amsart}

\usepackage[utf8]{inputenc}
\usepackage[T1]{fontenc}

\usepackage{amsmath}
\usepackage{amssymb}
\usepackage{amsfonts}
\usepackage{graphicx}
\usepackage{amsthm}
\usepackage{enumerate}
\usepackage{lscape}
\usepackage{dsfont}
\usepackage{color}
\usepackage{mathtools}

\usepackage{setspace}
\newcommand{\R}{\mathds{R}}

\newcommand{\CP}{\mathds{C}\mathrm{P}}
\newcommand{\N}{\mathds{N}}

\newcommand{\Ric}{\mathrm{Ric}}

\newcommand{\C}{\mathds{C}}            
\newcommand{\de}{\partial}          

\newcommand{\K}{K\"{a}hler}

\newcommand{\OO}{{\mathcal{O}}}

\newcommand{\ov}[1]{\overline{#1}}

\newcommand{\om}{\omega}

\newcommand{\va}{\varphi}

\newcommand{\FR}{{\mathfrak{R}}}
\newcommand{\FF}{{\mathfrak{F}}}
\newcommand{\CG}{{\mathcal{G}}}

\newcommand{\CO}{{\mathcal{O}}}

\def\b{\beta}

\def\b1{{\rm id}}

\newtheorem{theor}{Theorem}[section]
\newtheorem{prop}[theor]{Proposition}

\newtheorem{lem}[theor]{Lemma}
\newtheorem{cor}[theor]{Corollary}

\newtheorem{remark}{Remark}

\begin{document}

\title[K\"{A}hler immersions of K\"{A}hler-Ricci solitons into complex space forms]{K\"{A}hler immersions of K\"{A}hler-Ricci solitons into  definite or indefinite  complex space forms}

\author{Andrea Loi}
\address{(Andrea Loi) Dipartimento di Matematica \\
         Universit\`a di Cagliari (Italy)}
         \email{loi@unica.it}

\author{Roberto Mossa}
\address{(Roberto Mossa) 
Instituto de Matem\'atica e Estatistica  \\
         Universidade de S\~ao Paulo (Brasil)}
         \email{robertom@ime.usp.br}

\thanks{
The first author  supported  by Prin 2015 -- Real and Complex Manifolds; Geometry, Topology and Harmonic Analysis -- Italy, by INdAM. GNSAGA - Gruppo Nazionale per le Strutture Algebriche, Geometriche e le loro Applicazioni,  by STAGE - Funded by Fondazione di Sardegna and Regione Autonoma della Sardegna and by KASBA- Funded by Regione Autonoma della Sardegna.\\
\hspace*{1.3ex}The second author was supported by
FAPESP (grant: 2018/08971-9)}


\begin{abstract}
Let   $(g, X)$ be  a  \K--Ricci soliton on a complex manifold $M$.
We prove that if the \K\ manifold $(M, g)$ can be \K\ immersed into a  definite or indefinite complex space form then  $g$ is Einstein.   Notice that there is  no  topological assumptions on the manifold $M$ and the \K\ immersion is not required to be injective. Our result  extends the result obtained in  \cite{BG} 
asserting   that  a KRS  on  a compact \K\  submanifold $M\subset \CP^N$  
 which is  a complete intersection  is KE. 
\end{abstract}
 
\maketitle

\tableofcontents  

\section{Introduction}
A {\em Kahler-Ricci soliton} (KRS) on a complex manifold $M$ is a pair $(g, X)$ consisting of a Kähler metric $g$ and a
holomorphic vector field $X$, called the {\em solitonic vector field}, such that 
\begin{equation}\label{eqkrsg}
\Ric_{g}=\lambda g+L_{X}g
\end{equation}
for some $\lambda \in \mathbb{R}$,
where $\Ric_{g}$ is the Ricci tensor of the metric $g$ and $L_Xg$ denotes the Lie derivative of $g$ with respect to $X$, i.e.
$$(L_Xg)(Y, Z)=X\left( g(Y, Z)\right)-g([X, Y], Z)-g(Y, [X, Z]),$$
for $Y$ and $Z$ vector fields on $M$.
A  \K\ metric $g$ satisfying \eqref{eqkrsg}
gives rise to special solutions of the \K-Ricci flow (see e.g. \cite{ch}),
namely they evolve under biholomorphisms.
KRS generalize \K--Einstein (KE) metrics. Indeed any
KE metric $g$ on a complex manifold $M$ gives rise to a
trivial KRS by choosing $X = 0$ or $X$ Killing with
respect to $g$. Obviously if the automorphism group of $M$ is
discrete then a \K--Ricci soliton $(g, X)$ is nothing but a
KE metric $g$.

The first examples of non-Einstein compact KRS go back to the
constructions of N. Koiso \cite{KOISO} and independently H. D. Cao \cite{CAO1} of \K\ metrics on certain
$\C P^1$-bundles over $\C P^n$.
After that, X. J. Wang and X. Zhu \cite{WAZHU} proved the existence of
a KRS on any compact toric Fano manifold, and this result was later generalized by F. Podestà
and A. Spiro
\cite{POSPIRO} to toric bundles over generalized flag manifolds.
The reader is referred to  \cite{tiansol1}, \cite{tiansol2},
\cite{tiansol3} for the existence and uniqueness of \K--Ricci
solitons on compact manifolds and to \cite{FI} for the noncompact
case.

In this paper we address the problem of studying   KRS which can be  \K\ immersed  into a definite or indefinite 
(finite dimensional) complex space form $(S, g_c)$
of constant holomorphic sectional curvature $2c$.
The main result of the paper is the following theorem which asserts that such a KRS is trivial.
\begin{theor}\label{mainteor}
Let $(g, X)$  be a KRS on complex manifold $M$.
If $(M, g)$ can be  \K\ immersed   into a definite or indefinite complex space form $(S, g_c)$
then $g$ is KE. Moreover, its  Einstein constant is a rational multiple of $c$.
\end{theor}

It is worth pointing out  that in our theorem there  are no  topological assumptions on the manifold $M$ and the \K\ immersion is not required to be injective.
Notice also that our result is new even if we are in the realm of algebraic geometry, namely   when one assumes that  $M$ is compact, the ambient complex space form 
is the complex projective space (equipped with the Fubini--Study metric of constant holomorphic sectional curvature $4$) and  that  the immersion is an embedding. Indeed our result thereby extends the result obtained by A. Gori and L. Bedulli \cite{BG} 
asserting   that  a KRS  on  a compact \K\  submanifold $M\subset \CP^N$  
 which is  a complete intersection  is KE 
(and hence by a deep result of Hano \cite{HANO} $M$ turns out to be the quadric or a  complex projective space totally geodesically embedded in $\C P^N$). 
The reader is also referred to \cite{BGremar} where the condition  on complete intersection is replaced by the more general assumption that the  \K\ embedding has rational Gauss map.

By combining  well-known results on KE immersions into the complex projective space  in codimension one and two due to S. S. Chern \cite{ch} and  K. Tsukada \cite{ts} respectively,   we obtain the following corollary of Theorem \ref{mainteor}.

\begin{cor}
Let $(g, X)$ be a KRS on a $n$-dimensional complex manifold $M$.
If $(M, g)$ can be \K\ immersed into $\C P^{n+k}$ with $k\leq 2$,
then $M$ is either an open subset  of the complex quadric  or an open subset of $\C P^n$
totally geodesically embedded in $\C P^{n+k}$.
\end{cor}

Since a  rotation invariant  KE metric on a complex manifold of dimension $\geq 3$, which admits a \K\ immersion into a complex projective space in such a way that 
 the codimension is   $\leq 3$,
is forced to be the Fubini-Study metric (see \cite[Theorem 1.2]{SALIS} for a proof) we also get:

\begin{cor}
Let $(g, X)$ be a KRS on a $n$-dimensional complex manifold $M$ with $n\geq 3$.
Assume that $g$ is rotation invariant and  that $(M, g)$ can be \K\ immersed into $\C P^{n+k}$ with $k\leq 3$.
Then $(M, g)$ is an open subset of $\C P^n$
totally geodesically embedded in $\C P^{n+k}$.
\end{cor}

Theorem \ref{mainteor}  also yields the following result which can be deduced by D. Hulin's   theorem \cite[Prop. 4.5]{HUfr} 
on the extension of germs of   KE projectively induced metrics.

\begin{cor}
Let $(g, X)$ be a projectively induced  shrinking KRS (i.e. $\lambda>0$ in \eqref{eqkrsg}) on a complex manifold $M$. 
Then $g$ extends to a projectively induced  KE metric $\hat g$, with positive and rational Einstein constant $\lambda$,  on a  compact 
 complex  manifold $\hat M$. 
\end{cor}

Notice that  D. Hulin \cite{HU} shows that all the compact  KE submanifolds of the complex projective space has necessary positive (rational) Einstein constant
and it is conjecturally true (see e.g. \cite{LZ})  that all such manifolds are flag manifolds.

By combining Theorem \ref{mainteor} with M. Umehara \cite{UmearaE} results on  KE manifolds immersed into the 
flat or complex hyperbolic space  we get:
\begin{cor}
Let $(g, X)$ be a KRS on a complex manifold $M$.
A \K\ immersion of  $(M, g)$ into a definite  complex space form
of nonpositive holomorphic sectional curvature is totally geodesic.
\end{cor}

As a special case  of the last part of Theorem \ref{mainteor} in the  indefinite case we get:  
\begin{cor}\label{corfin}
The Einstein constant of a KE submanifold of an indefinite complex space form $(S, g_c)$ is a  rational multiple of $c$.
\end{cor}
When the ambient 
space is the  indefinite complex projective space this  corollary 
can be considered an  extension 
of D. Hulin result \cite[Prop. 5.1]{HUfr} on  the rationality of the  Einstein constant of a projectively induced KE metric.
Moreover, to the best of authors' knowledge,
the classification of the  KE submanifolds of the indefinite complex  hyperbolic space is missing.
Even in  the  codimension one case (where such classification is known (\cite[Theorem 3.2.4]{RO})
Umehara's theorem
does not hold (there exists non totally geodesic KE submanifolds of the indefinite complex hyperbolic space).  Thus, Corollary \ref{corfin} seems to be a novelty also in the case of KE submanifolds of the indefinite complex  hyperbolic space.

\vskip 0.3cm

The proof of Theorem \ref{mainteor} is based on Theorem \ref{lem} (see next section)  which describes  some properties of Umehara algebra and its field of fractions.
Section 3 is dedicated to the proof of Theorem \ref{mainteor}.

\vskip 0.3cm
We would like to thank Yuan Yuan for his interest in our work 
and for stimulating discussions about Umehara's algebra.

\section{Umehara algebra and its field of fractions}
Let $M$ be a complex manifold. Fix a point  $p\in M$ and let 
$\OO_p$ be the algebra of germs of holomorphic functions around $p$.
Denote  by $\R_p$ the germs of real numbers. 
The Umehara algebra (see  \cite{UmearaD}) is defined  to be the $\R$-algebra $\Lambda_p$ generated
by the elements of the form  $h\bar k+\bar h k$, for $h, k\in \OO_p$. 
Umehara algebra has been an important tool in the study of relatives \K\ manifolds (see \cite{diloi, UmearaC, UmearaE, UmearaD,  CDY}).

Let
$$\hat\OO_p=\left\{\alpha =(\alpha_1,   \dots ,\alpha_m) \ | \   \alpha_j\in\OO_p, \alpha_j(p)=0 , \forall j=1, \dots , m, m\geq 1\right\}.$$
For  $\alpha =(\alpha_1,  \dots ,\alpha_m)\in \hat\OO_p$
and $\ell\in\N$ such that $\ell\leq |\alpha|:=m$
we set
$$\langle \alpha, \alpha\rangle_{\ell}(z)=\sum_{j=1}^\ell|\alpha_j(z)|^2- \sum_{k=\ell+1}^{|\alpha|}|\alpha_k(z)|^2.$$
Since $h\bar k+\bar h k=|h+k|^2-|h|^2-|k|^2$ it is not hard to see (see \cite{UmearaD} for details) 
that each $f \in \Lambda_p$  can be written  as
$$f=h+\bar h+\langle \alpha, \alpha\rangle_{\ell}$$
for some $h\in \OO_p$, $\alpha =(\alpha_1,  \dots ,\alpha_m)\in \hat\OO_p$, 
$\ell\leq |\alpha|$
and such that  $\alpha_1, \dots , \alpha_m$  are linearly independent over $\C$.

Consider  the $\R$-algebra $\tilde\Lambda_p\subset \Lambda_p$ given by
\begin{equation}\label{Lambda0}
\tilde\Lambda_p=\left\{a+ \langle \alpha, \alpha\rangle_{\ell} \mid a\in \R_p,\ \alpha\in \hat\OO_p,\  \ell\leq |\alpha|\right\}.
\end{equation}
Notice that the germ of the real part of a nonconstant  holomorphic function 
$h\in\OO_p$ belongs to 
$\Lambda_p$ but not to 
$\tilde\Lambda_p$. 

The key element in the proof of Theorem \ref{mainteor} is the following Theorem \ref{lem} whose proof is inspired  by the work X. Huang and Y. Yuan  \cite{HY} (see also \cite{YUANYUAN}).

\begin{theor}\label{lem}
Let $\tilde K_p$ be the field of fractions of  $\tilde\Lambda_p$.
Let $\mu$ be  a  real number and  $g=\frac{b+ \langle \beta, \beta\rangle_{q}}{c+ \langle \gamma, \gamma\rangle_{r}}\in \tilde K_p$, then
\begin{equation}\label{fundeq}
e^{g}\not\in {\tilde\Lambda_p}^\mu\tilde K_p\setminus \R_p
\end{equation}
where ${\tilde\Lambda_p}^\mu\tilde K_p=\left\{f^\mu h\mid f\in\tilde\Lambda_p, \  h\in \tilde K_p \right\}$.
\end{theor}

\begin{remark}\rm\label{remarrat}
 Theorem \ref{lem} extends 
\cite[Theorem 2.1, part (ii)]{CDY}  which asserts that $e^g\not\in \tilde K_p\setminus \R_p$,
for  $g\in \tilde\Lambda_p$. Moreover \cite[Theorem 2.1, part (iii)]{CDY} shows that if $f^\alpha$
 belongs to $\tilde K_p\setminus \R_p$
 for some real number $\alpha$ and  $f\in \tilde\Lambda_p$,
then $\alpha$ is rational. This result will be used to prove the rationality of the Einstein constant in Theorem \ref{mainteor}.
\end{remark}

The following lemma is a key ingredient  in the proof of Theorem \ref{lem}.

\begin{lem}\label{lemHY}(\cite[Lemma 2.2]{HYbook})
Let $V \subset \mathbb{C}^{\kappa}$ be a connected open set. Let $H_{1}\left(\xi_{1},\dots,\xi_{\kappa}\right)$, $\dots$, $H_{K}\left(\xi_{1}, \ldots, \xi_{\kappa}\right)$ and $H\left(\xi_{1}, \ldots, \xi_{\kappa}\right)$ be holomorphic Nash algebraic functions on $V$. Assume that
\begin{equation*}
\exp ^{H\left(\xi_{1}, \ldots, \xi_{\kappa}\right)}=\prod_{\alpha=1}^{K}\left(H_{k}\left(\xi_{1}, \ldots, \xi_{\kappa}\right)\right)^{\mu_{\alpha}}, \xi \in V
\end{equation*}
for certain real numbers $\mu_{1}, \ldots, \mu_{K} .$ Then $H\left(\xi_{1}, \ldots, \xi_{\kappa}\right)$ is constant.
\end{lem}

\begin{proof}[Proof of Theorem \ref{lem}]
Assume that there exists
$f=a+ \langle \alpha, \alpha\rangle_{\ell}\in \tilde\Lambda_p$ and  $h=\frac{d+ \langle \delta, \delta\rangle_{u}}{e+ \langle \epsilon, \epsilon\rangle_{v}}\in\tilde K_p$ such that
\begin{equation}\label{egz}
e^{g(z)}=e^\frac{b+ \langle \beta(z), \beta(z)\rangle_{q}}{c+ \langle \gamma(z)), \gamma(z)\rangle_{r}}=f^\mu(z)h(z)=\left[a+ \langle \alpha(z), \alpha(z)\rangle_{p}\right]^\mu \frac{d+ \langle \delta(z), \delta(z)\rangle_{u}}{e+ \langle \epsilon(z), \epsilon(z)\rangle_{v}}.
\end{equation}
By renaming the functions involved in \eqref{egz} we can write 
$$
S=\left\{\varphi_1, \dots , \varphi_s\right\}=\left\{\alpha_1,\dots,\alpha_{|\alpha|} ,\beta_1,\dots,\beta_{|\beta|} , \dots, \epsilon_1,\dots,\epsilon_{|\epsilon|} \right\}.
$$ 
Let $D$ be an open neighborhood  of the origin of $\C^n$ on which each $\varphi_j$ is defined.
Consider   the field  $\mathfrak R$ of rational function on $D$ and its  field extension 
$
\mathfrak F = \mathfrak R \left(S\right),
$
namely, the smallest subfield of the field of the  meromorphic functions on $D$, containing  rational functions and the elements of $S$. Let $l$ be the transcendence degree  of the field extension $\FF/\FR$. 
If $l=0$, then each element in $S$ is holomorphic  Nash algebraic and hence $g$ is forced to be constant  by Lemma \ref{lemHY}. Assume then that $l>0$.
Without loss of generality we can assume that  $\CG=\left\{\varphi_1,\dots,\varphi_l\right\}\subset S$  is a maximal algebraic independent subset over $\FR$. Then   there exist minimal polynomials $P_j\left(z, X,Y\right)$, $X=\left(X_1,\dots,X_l\right)$,  such that 
$$
P_j\left(z,\Phi(z), \va_j(z)\right)\equiv 0, \ \forall j=1, \dots ,s,
$$
where $\Phi(z)=\left(\va_1(z),\dots,\va_l(z)\right)$.

Moreover, by  the definition of minimal polynomial 
 $$
 \frac{\de P_j\left(z,X,Y\right)}{\de Y}\left(z,\Phi (z),\va_j(z)\right)\not\equiv 0, \ \forall j=1, \dots ,s.
 $$
  on $D$.
 Thus, by the algebraic version of the existence and uniqueness part of the implicit function theorem, there exist a connected open subset $U\subset D$ with $p\in \ov U$ and Nash algebraic functions 
 $\hat \va_j(z,X)$,  defined in a neighborhood $\hat U$ of $\left\{(z, \Phi(z)) \mid z \in U \right\}\subset \C^n \times \C^{l}$, such that
$$
\va_j(z)=\hat \va_j\left(z,\Phi(z)\right), \ \forall j=1, \dots ,s.
$$
for any $z\in U$. Denoting $\hat\varphi(z, X)=(\hat\varphi_1 (z, X), \dots , \hat\varphi_s(z, X))$ we can write 
\begin{equation}\label{varie}
\va (z)=\hat \va \left(z,\Phi (z)\right)= 
\left(\hat \alpha\left(z,\Phi (z)\right), \dots ,\hat\epsilon\left(z,\Phi (z)\right)\right),
\end{equation}
where
$\varphi=(\varphi_1, \dots , \varphi_s)$
and $\hat\alpha (z, X), \dots ,\hat\epsilon (z, X)$  are vector-valued  holomorphic Nash algebraic  functions on  $\hat U$ such that 
$
\alpha(z)=\hat \alpha \left(z,\Phi(z)\right),\dots  ,\epsilon(z)=\hat \epsilon \left(z,\Phi(z)\right)$.

Consider the function
\begin{equation*}
\Psi(z,X,w):=\frac{b+ \langle \hat\beta(z, X), \beta(w)\rangle_{q}}{c+ \langle \hat\gamma(z, X), \gamma(w)\rangle_{r}}
-\mu\log \left[a+ \langle \hat\alpha(z, X), \alpha(w)\rangle_{\ell}\right]
-\log\left[\frac{d+ \langle \hat\delta (z, X), \delta (w)\rangle_{u}}{e+ \langle \hat\epsilon (z, X), \epsilon (w)\rangle_{v}}\right],
\end{equation*}
where, for  $\alpha (z)=(\alpha_1(z),  \dots ,\alpha_m(z))$ and corresponding 
$\hat\alpha (z,X)=(\hat\alpha_1(z, X),  \dots ,\hat\alpha_m(z, X))$
we mean
$$\langle \hat\alpha(z, X), \alpha(w)\rangle_{\ell}=\sum_{j=1}^\ell\hat\alpha_j(z,X)\alpha_j(\overline w)- \sum_{k=\ell+1}^{|\alpha|}\hat\alpha_k(z, X)\alpha_k(\overline w)$$
(and similarly with the other terms).
By shrinking $U$ if necessary we can assume $\Psi(z,X,w)$  is defined on  $\hat U \times U$.
We claim that $\Psi(z,X,w)$ vanishes identically on  this set.
Since $\varphi_j(p)=0$ for all $j=1, \dots ,s$ and $p\in \ov U$, it follows by \eqref{egz}
that  $\Psi(z,X,0)\equiv 0$. Hence, in order to prove the claim, it is enough to show that 
$(\de_w \Psi)(z,X,w)\equiv 0$ for all $w\in U$.  Assume, by contradiction, that there exists $w_0\in U$ such that $(\de_w \Psi)(z,X,w_0)\not\equiv 0$. Since $(\de_w \Psi)(z,X,w_0)$ is Nash algebraic in $(z,X)$ there exists a holomorphic polynomial $P(z,X,t)=A_d(z,X)t^d+\dots+A_0(z,X)$ with $A_0(z,X)\not\equiv 0$ such that  $P(z,X,(\de_w \Psi)(z,X,w_0))=0$. 
Since,  by \eqref{egz}  and  \eqref{varie} we have
 $\Psi(z,\Phi (z),w)\equiv 0$ we get $(\de_w \Psi)(z,\Phi (z),w)\equiv 0$. Thus  $A_0(z,\Phi (z))\equiv 0$
 which contradicts  the fact that $\va_1(z),\dots,\va_l(z)$ are algebraic independent over $\FR$. Hence $(\de_w \Psi)(z,X,w_0)\equiv 0$ and the claim is proved.

Therefore
$$e^{\frac{b+ \langle \hat\beta(z, X), \beta(w)\rangle_{q}}{c+ \langle \hat\gamma(z, X), \gamma(w)\rangle_{r}}}=\left[a+ \langle \hat\alpha(z, X), \alpha(w)\rangle_{\ell}\right]^\mu
\left[\frac{d+ \langle \hat\delta (z, X), \delta (w)\rangle_{u}}{e+ \langle \hat\epsilon (z, X), \epsilon (w)\rangle_{v}}\right],$$
for every $(z,X,w)\in \hat U \times U$. By fixing  $w\in U$ and  applying Lemma \ref{lemHY} we deduce that 
$\frac{b+ \langle \hat\beta(z, X), \beta(w)\rangle_{q}}{c+ \langle \hat\gamma(z, X), \gamma(w)\rangle_{r}}$
is constant in $(z,X)$.
Thus, by evaluating at $X=\Phi(z)$ one obtains  that 
$\frac{b+ \langle \beta(z), \beta(w)\rangle_{q}}{c+ \langle \gamma(z), \gamma(w)\rangle_{r}}$
is constant for fixed $w$ forcing
$g(z)=\frac{b+ \langle \beta(z), \beta(z)\rangle_{q}}{c+ \langle \gamma(z), \gamma(z)\rangle_{r}}$ to be constant  for all $z$. The proof of the theorem is complete.
\end{proof}

\section{Proof of Theorem \ref{mainteor}}
In the proof of Theorem \ref{mainteor} we  also need the concept of  diastasis function and Bochner's coordinates briefly recalled below. The reader is referred either to the celebrated work of Calabi \cite{Cal} or to \cite{LoiZedda-book}
for details.

Given a complex manifold $M$ endowed with a real analytic 
\K \ metric $g$ (notice that a \K\ metric induced by a complex space form is real analytic),
Calabi introduced,
in a neighborhood of a point 
$p\in M$,
a very special
\K\ potential
$D^g_p$ for the metric
$g$, which 
he christened 
{\em diastasis}.
Among all the potentials the diastasis
is characterized by the fact that 
in every coordinate system 
$\{z_1, \dots , z_n\}$ centered in $p$
$$D^g_p(z, \bar z)=\sum _{|j|, |k|\geq 0}
a_{jk}z^j\bar z^k,$$
with 
$a_{j 0}=a_{0 j}=0$
for all multi-indices
$j$.
One of the main feature of Calabi's diastasis function is its hereditary property: {\em if $\varphi:M\rightarrow S$
is a \K\ immersion from a \K\ manifolds $(M, g)$ into a definite or indefinite complex space form $(S, g_c)$ of constant holomorphic sectional curvature $2c$,
then 
$D^g_p=\varphi^*(D^{g_c}_{\varphi (p)})$
for all  $p\in M$}.

More generally, in the rest of the paper   we say that  a real analytic function defined on a neighborhood   $U$ of   a point $p$
of a complex manifold $M$ is of  {\em diastasis-type} if 
in one (and hence any) coordinate system 
$\{z_1, \dots , z_n\}$ centered at $p$ its  expansion in $z$ and $\bar z$ 
does not contains non constant purely holomorphic or anti-holomorphic terms (i.e. of the form $z^{j}$ or $\bar{z}^{j}$ with  $j > 0$). 
The following simple  remarks will be used in the proof of Theorem \ref{mainteor}.

\begin{remark}\label{keyremark}\rm
The  following facts holds true.
\begin{itemize}
\item [(a)] A real-analytic function $g$ is of  diastasis-type if and only if  $e^g$ is of  diastasis-type. 
\item [(b)] A function $f\in\Lambda_p$ (resp. $K_p$) belong to $\tilde\Lambda_p$
(resp. $\tilde K_p$) if and only if  $f$ is of diastasis-type, where   $\tilde K_p$ is the field of fractions of the algebra $\tilde\Lambda_p$ defined by \eqref{Lambda0}.
\end{itemize}
\end{remark}

In a neighborhood of 
$p\in M$ one can  find local 
(complex) coordinates
such that
$$D^g_p(z, \bar z)=|z|^2+
\sum _{|j|, |k|\geq 2}
b_{jk}z^j\bar z^k,$$
where 
$D^g_p$ is the diastasis
relative to $p$.
These coordinates,
uniquely defined up
to a unitary transformation,
are called 
{\em the Bochner   or normal coordinates}
with respect to the point $p$ (cfr. \cite{note, boc, Cal}).

The following proposition, interesting on its own sake, will be used in  the proof of Theorem \ref{mainteor}.\begin{prop}\label{mainprop}
Let $(M, g)$ be a \K\ manifold which  can be  \K\ immersed   into an $N$-dimensional definite or indefinite complex space form $(S, g_c)$
of constant holomorphic sectional curvature $2c$.
Let $p\in M$ and  $\{z_1, \dots , z_n\}$ Bochner's coordinates  in a neighborhood $U$ of a point $p\in M$ where the diastasis $D^g_p$ is defined.
Then 
\begin{equation}\label{detko}
\det\left[\frac{\partial^{2} D^g_{p}}{\partial z_{a} \partial\bar z_{\beta}}\right]\in \tilde K_p,
\end{equation}
\begin{equation}\label{Dpc=0}
D^g_p\in\tilde\Lambda_p,\  \mbox{if}\  c=0
\end{equation}
and
\begin{equation}\label{Dpc=1}
e^{\frac{c}{2}D^g_p}\in\tilde\Lambda_p,\  \mbox{if}\  c\neq 0.
\end{equation}
\end{prop}

In the proof of the proposition we need  the following result.
\begin{lem}(\cite[Lemma 2.2]{UmearaE} )\label{umeharaimp}
Let $M$ be  an $n$-dimensional complex manifold and  $p\in M$. For any  $f$ in the Umehara algebra $\Lambda_p$
and for any system of complex coordinates $\{z_1, \dots z_n\}$ around $p$ one has:
$$f^{n+1}\frac{\partial^2\log f}{\partial z_{\alpha}\partial\bar z_{\beta}}\in \Lambda_p. \ \forall \alpha, \beta =1, \dots , n.$$
\end{lem}

\begin{proof}[Proof of Proposition \ref{mainprop}]
Let  $\varphi:M\rightarrow S$ be a \K\ immersion into  $(S, g_c)$, i.e. $\varphi$ is holomorphic and $\varphi^*g_c=g$.
If one assumes that   $(S, g_c)$ is  complete and simply-connected one has the corresponding three cases, depending on the sign of $c$:

- for $c=0$, $S=\C^N$ and  $g_{0}$ is the flat metric with   associated \K\ form
\begin{equation}\label{omega0} 
\omega_{0}=\frac{i}{2}\partial\bar\partial |z|_s^2,
\end{equation}
where we set 
$$|z|_s^2=|z_1|^2+\cdots +|z_s|^2- |z_{s+1}|^2-\cdots -|z_{N}|^2, \ 0\leq s\leq N.$$

- for $c>0$, $S=\C P_s^N$ is the open submanifold 
 $$
  \left\{[Z_0, \dots , Z_s, Z_{s+1}, \dots , Z_N] \in\C P^N \mid |Z_0|^2 + \dots + |Z_s|^2- |Z_{s=1}|^2 - \dots - |Z_N|^2 > 0\right\}
 $$
of $\C P^N$ and $g_c$  is the  metric with associated  \K\ form  $\omega_c$ 
 given in the affine  chart $U_0=\{[Z_0, \dots ,Z_N] \ | \ Z_0\neq 0\}$ with coordinates $z_j=\frac{Z_j}{Z_0}$  as:
\begin{equation}\label{omegaproj} 
\omega_c=\frac{i}{c}\partial\bar\partial\log (1+|z|_s^2);
\end{equation}

- for $c<0$, $S=\C H_s^N$ is open subset of  $\C^N$
 given by $\{z\in \C^N \ | \ |z|_s<1\}$
  with the metric $g_c$ with associated  \K\ form 
 \begin{equation}\label{omegahyp} 
\omega_c=\frac{i}{c}\partial\bar\partial\log (1-|z|_s^2).
\end{equation}

Let $$\varphi_{|U}:U\rightarrow \C^N, z=(z_1, \dots z_n)\mapsto (\varphi_1(z), \dots , \varphi_N(z))$$
where $\varphi_j\in \CO_p$ and $\varphi_j(p)=0$, $j=1, \dots ,N$,
be the local expression of $\varphi$ in Bochner's  coordinates $\{z_1, \dots , z_n\}$.

In order to prove \eqref{detko} we first consider the case $c=0$.  By the hereditary property of the diastasis function and \eqref{omega0}
we  have 
\begin{equation}\label{Dpflat}
D^g_p =\sum_{i=1}^N \left|\va_i\right|_s^2.
\end{equation}
 Thus the  function
 $\det\left[\frac{\partial^{2} D^g_{p}}{\partial z_{a} \partial\bar z_{\beta}}\right]$ is finitely generated by holomorphic or anti-holomorphic functions around $0$. Furthermore it is real valued, since the matrix  $\frac{\partial^{2} D_{p}}{\partial z_{a} \bar\partial z_{\beta}}$ is Hermitian. We conclude that 
 $\det\left[\frac{\partial^{2} D^g_{p}}{\partial z_{a} \partial\bar z_{\beta}}\right]\in \Lambda_p$. 
 It is easy to check that in  Bochner's coordinates the function $\det\left[\frac{\partial^{2} D^g_{p}}{\partial z_{a} \partial\bar z_{\beta}}\right]$ is of diastasis-type. Thus 
 $\det\left[\frac{\partial^{2} D^g_{p}}{\partial z_{a} \partial\bar z_{\beta}}\right]\in \tilde\Lambda_p\subset \tilde K_p$. 
 
Let us now consider  the case $c\neq 0$. Again by the hereditary property of the diastasis and by \eqref{omegaproj} and \eqref{omegahyp} we can write
\begin{equation}\label{Dpgen}
D^g_p = \frac{2}{c}\log \left(1+\frac{c}{|c|}
\sum_{i=1}^N \left|\va_i\right|_s^2\right).
\end{equation}
It follows by Lemma \ref{umeharaimp} applied to 
$$e^{\frac{c}{2}D^g_p}=\left(1+\frac{c}{|c|}
\sum_{i=1}^N \left|\va_i\right|_s^2\right) \in  \tilde\Lambda_p\subset \Lambda_p$$  that 
$$
e^{(n+1)\frac{c}{2}D^g_p}\det  \left[\frac{\partial^{2} (\frac{c}{2}D^g_{p})}{\partial z_{a} \partial\bar z_{\beta}}\right]=\left(\frac{c}{2}\right)^ne^{(n+1)\frac{c}{2}D^g_p}\det\left[\frac{\partial^{2} D^g_{p}}{\partial z_{a} \partial\bar z_{\beta}}\right]  \in\Lambda_p, 
$$
and hence
$$\det\left[\frac{\partial^{2} D^g_{p}}{\partial z_{a} \partial\bar z_{\beta}}\right]\in K_p.$$
Also in this case  it not hard to see
that in Bochner's coordinates  $\det\left[\frac{\partial^{2} D^g_{p}}{\partial z_{a} \partial\bar z_{\beta}}\right]$ 
is of diastasis-type and hence \eqref{detko} readily follows.

Finally, the proofs of  \eqref{Dpc=0} and \eqref{Dpc=1} follow by   \eqref{Dpflat} and   \eqref{Dpgen}.
\end{proof}
\begin{proof}[Proof of Theorem \ref{mainteor}]
Let us start to write down equation \eqref{eqkrsg} in local complex   coordinates $\{z_1, \dots , z_n\}$ in a neighborhood $U$ of a point $p\in M$ where the diastasis $D^g_p$ for the metric $g$  is defined.
Since the solitonic vector field $X$ can be assumed to be  the real part of a holomorphic vector field, we can write 
$$
X=\sum_{j=1}^{n}\left(f_{j} \frac{\partial}{\partial z_j}+\bar{f}_{j}\frac{\partial}{\partial\bar z_j}\right)
$$
for some holomorphic functions $f_{j}, j=1, \ldots, n$, on $U$.

Thus, by the definition of Lie derivative, after a 
straightforward computation we can write on $U$
\begin{equation}\label{LXomega}
L_X\omega=\frac{i}{2}\partial\bar\partial f_X.
\end{equation}
where $\omega$ is the \K\ form associated to $g$ and 
\begin{equation}\label{fXlocal}
f_X=\sum_{j=1}^nf_j\frac{\partial D^g_p}{\partial z_j}+\bar f_j\frac{\partial D^g_p}{\partial \bar z_j}.
\end{equation}

Notice that equation  \eqref{eqkrsg} is equivalent to
\begin{equation}\label{eqkrsom}
\rho_{\om}=\lambda \om+L_{X} \om,
\end{equation}
where  $\rho_{\om}$ the  Ricci form of $\omega$.

Since,   $\omega=\frac{i}{2}\partial\bar\partial D^g_p$ and  $\rho_\omega=-i\partial\bar\partial\log \det\left[\frac{\partial^{2} D^g_{p}}{\partial z_{a} \partial\bar z_{\beta}}\right]$ on $U$,
the local expression of  the  KRS equation \eqref{eqkrsom} is
$$-i\partial\bar\partial\log \det\left[\frac{\partial^{2} D^g_{p}}{\partial z_{a} \partial\bar z_{\beta}}\right]=\lambda \frac{i}{2}\partial\bar\partial D^g_p+\frac{i}{2}\partial\bar\partial f_X,$$
and by the $\partial\bar\partial$-Lemma one has
\begin{equation}\label{detee}
\det\left[\frac{\partial^{2} D^g_{p}}{\partial z_{a} \partial\bar z_{\beta}}\right]=e^{-\frac{\lambda}{2}D^g_p-\frac{f_X}{2}+h+\bar h},
\end{equation}
for a holomorphic function $h$ on $U$.

We treat the two cases $c=0$ and $c\neq 0$ separately.
If $c=0$,  combining \eqref{Dpc=0}, \eqref{Dpflat} and  \eqref{fXlocal} we get that 
$$g_1:=-\frac{\lambda}{2}D^g_p-\frac{f_X}{2}+h+\bar h\in\Lambda_p.$$
By   \eqref{detko} in Proposition \ref{mainprop},   $\det\left[\frac{\partial^{2} D^g_{p}}{\partial z_{a} \partial\bar z_{\beta}}\right]\in\tilde K_p$   and hence  \eqref{detee} gives 
$e^{g_1}\in\tilde K_p$. In particular $e^{g_1}$  and so, by (a)  of Remark \ref{keyremark},  $g_1$  is of diastasis-type.  It follows then by (b) of Remark \ref{keyremark} that  $g_1\in\tilde\Lambda_p$. Thus  Theorem \ref{lem}  (with $\mu=0$) forces  $g_1$ to be a constant. Hence  $\det\left[\frac{\partial^{2} D^g_{p}}{\partial z_{a} \partial\bar z_{\beta}}\right]$ is a constant and  so $g$ is Ricci flat.

If $c\neq 0$,  by combining \eqref{Dpgen} and  \eqref{fXlocal} one easily sees 
that  
$$g_2:=-\frac{f_X}{2}+h+\bar h\in K_p.$$ 
By  \eqref{detko}, \eqref{Dpc=1} 
and \eqref{detee} one deduces that
\begin{equation}\label{eg2}
e^{g_2}=
\left[e^{\frac{c}{2}D^g_p}\right]^{\frac{\lambda}{c}}\det\left[\frac{\partial^{2} D^g_{p}}{\partial z_{a} \partial\bar z_{\beta}}\right]\in {\tilde\Lambda_p}^\mu\tilde K_p,\ \mu=\frac{\lambda}{c}.
\end{equation}
On  the one hand  \eqref{eg2} shows that $e^{g_2}$ and hence (by (a) of Remark \ref{keyremark})  $g_2$  is of diastasis-type  and so, by  (b) of Remark \ref{keyremark},
$g_2\in\tilde K_p$. 
On the other  hand  \eqref{eg2} together  with 
Theorem \ref{lem}  force $g_2$  to be a  constant and so $f_X$ is the real part of a holomorphic function.
Therefore, by  \eqref{LXomega} and  \eqref{eqkrsom} the metric   $g$ is KE.
Moreover $\frac{\lambda}{c}$ is  forced to be rational by the second part of Remark \ref{remarrat},
completing the proof of  Theorem \ref{mainteor}.
\end{proof}

\end{document}